\documentclass[12pt]{amsart}
\usepackage{amsmath}
\usepackage{amssymb}
\usepackage{amsthm}
\usepackage{array}
\usepackage{xy}
\usepackage[pdftex]{graphicx}
\usepackage{hyperref}
\usepackage{color}
\usepackage{transparent}
\usepackage{latexsym}

\numberwithin{equation}{section}

\newcommand{\N}{\mathbb{N}}
\newcommand{\R}{\mathbb{R}}

\usepackage{a4wide,color,eucal,enumerate,mathrsfs}
\usepackage[normalem]{ulem}
\usepackage{amsmath,amssymb,amsfonts,amsthm,epsfig,graphicx}
\usepackage{esint}
\numberwithin{equation}{section}

\usepackage[latin1]{inputenc}

\newtheorem{theorem}{Theorem}[section]

\newtheorem{corollary}[theorem]{Corollary}
\newtheorem{lemma}[theorem]{Lemma}

\newtheorem{remark}[theorem]{Remark}

\newcommand{\ep}{\varepsilon}

\begin{document}
\title[Obstructions to regularity in the classical Monge problem]{Obstructions to regularity in the classical Monge problem}
\author[Maria Colombo]{Maria Colombo}
\author[Emanuel Indrei]{Emanuel Indrei}

\date{}

\maketitle

\makeatletter
\def\blfootnote{\xdef\@thefnmark{}\@footnotetext}
\makeatother

\blfootnote{This work was supported by NSF grant DMS-0932078, administered 
by the Mathematical Sciences Research Institute in Berkeley, California, while the authors were in residence at MSRI during the 2013 program ``Optimal Transport: Geometry and Dynamics."
The first author was also supported by the {\it Gruppo Nazionale per l'Analisi Matematica, la Probabilit\`a e le loro Applicazioni (GNAMPA)} of the Italian {\it  Istituto Nazionale di Alta Matematica (INdAM)} and by the {\it PRIN 2011 Calcolo delle variazioni} of the italian {\it Ministero dell'istruzione dell'Universit\`a e della Ricerca}.}

\def\signei{\bigskip\begin{center} {\sc Maria Colombo\par\vspace{3mm}
Scuola Normale Superiore\\
p.za dei Cavalieri 7, I-56126\\
Pisa, Italy\\
\
email:} {\tt maria.colombo@sns.it}
\end{center}}

\def\signdm{\bigskip\begin{center} {\sc Emanuel Indrei\par\vspace{3mm}
MSRI\\  
17 Gauss Way\\
Berkeley, CA 94720\\
email:} {\tt eindrei@msri.org }
\end{center}}

\begin{abstract}
We provide counterexamples to regularity of optimal maps in the classical Monge problem under various assumptions on the initial data. Our construction is based on a variant of the counterexample in \cite{LSW} to Lipschitz regularity of the monotone optimal map between smooth densities supported on convex domains. 
\end{abstract}

\section{Introduction}
The classical optimal transportation problem appears in a $1781$ article of Gaspard Monge \cite{GM}. In the modern formulation, one is given a cost function $c: \R^n \times \R^n \rightarrow [0,\infty]$ and a pair of probability densities $f$ and $g$ defined on two domains $\Omega \subset \R^n$, $\Omega^* \subset \R^n$, respectively. The objective is to minimize the functional 
\begin{equation} \label{min}
T \rightarrow \int_\Omega c(x,T(x)) f(x) \,dx
\end{equation}
among maps $T: \Omega \rightarrow \Omega^*$ which satisfy the ``push-forward" condition $T_{\#}f=g$, i.e. $$\int_A g(y)dy=\int_{T^{-1}(A)} f(x)dx  \hskip .1in \text{for all Borel sets $A \subset \Omega^*$.}$$      
A minimizer of \eqref{min} is commonly referred to as an ``optimal map" or ``optimal transport" for the corresponding cost function. In the classical problem, one selects the Monge cost, i.e. $c(x,y)=|x-y|$. The existence of optimal maps was addressed by many authors \cite{TW, CFM, EG, Sud, Am}. In general, minimizers are not unique; nevertheless, there is a unique optimal transport which is monotone on transfer rays \cite{FM}. The regularity of optimal maps is still widely open; the only known result is in $\R^2$: if the given densities are positive, continuous, and have compact, disjoint, convex support, then the monotone optimal map is continuous in the interior of the transfer rays \cite{FGP}. Recently, it was shown in \cite{LSW} that the monotone optimal transport between positive $C^\infty$ densities supported on the same bounded, convex domain fails to be Lipschitz continuous at interior points. In fact, the authors prove something stronger: namely, they construct an example in which the monotone optimal map does not belong to $C^{\frac{2}{3}+\ep}$ for any $\ep>0$.

In this paper we focus on examples relating the regularity of the initial data, which are assumed to be strictly positive and bounded on a convex set, with the regularity of the monotone optimal transport. As a starting point, the following example shows that in general, the monotone optimal map is not more regular than the initial data: consider $\Omega=\Omega^*=(0,1)\times(0,1)$ and suppose that $f$ and $g$ are bounded away from zero and infinity. Let $f_{x_1}(\cdot):=f(x_1, \cdot)$ and $g_{x_1}(\cdot):=g(x_1,\cdot)$, and write $F_{x_1}$ and $G_{x_1}$ for their primitives, respectively. If $F_{x_1}(0)=G_{x_1}(0)=0$, $F_{x_1}(1)=G_{x_1}(1)$, and $F_{x_1}(x_2) \geq G_{x_1}(x_2)$ for all $x_1, x_2$, then the optimal potential is $u(x_1,x_2)=-x_2$ and the monotone optimal transport is explicitly given by 
\begin{equation} \label{ex1}
T(x_1,x_2)=G_{x_1}^{-1}(F_{x_1}(x_2)). \footnote{This example was communicated to us by Filippo Santambrogio.} 
\end{equation}
We note that the primitives endow $T$ with more regularity than the densities in the $x_2$ variable; nevertheless, the map has the same regularity as $f$ and $g$ in the $x_1$ variable. For example, continuous densities may produce a monotone optimal map which is not H\"older continuous. Moreover, the same example shows that a discontinuous monotone optimal map may arise from discontinuous, bounded initial data (see also \cite[Example 4.5]{FGP}). 

Our main result constructs a family of examples which shed further light on the nature of the monotone optimal transport map, see Theorem \ref{thm:main}. More precisely, we build examples where the modulus of continuity of the monotone optimal transport is worse than the modulus of continuity of the densities. Indeed, some corollaries of our result include the following statements ($\Omega$ is the interior of a triangle, and the densities are positive on $\overline{\Omega}$):
\begin{eqnarray}
\label{eqn:c-infty-23}
f,g \in C^\infty (\Omega) &\not \Rightarrow& T \in C_{loc}^{0,\frac{2}{3} +\ep}(\Omega), \qquad \forall \text{$\ep>0$};
\\
\label{bdry}
f,g \in C^{\infty} (\overline{\Omega}) &\not \Rightarrow& T \in C^{0,\frac{1}{2} +\ep}(\overline{\Omega}), \qquad \forall  \text{$\ep>0$};
\\
\label{eqn:c21-12}
f,g \in C^{2,1} (\Omega) &\not \Rightarrow& T \in C_{loc}^{0,\frac{1}{2} +\ep}(\Omega), \qquad \forall  \text{$\ep>0$};
\\
\label{eqn:calpha-calpha}
f,g \in C^{0,\alpha}(\Omega) &\not \Rightarrow& T \in C_{loc}^{0,\frac{\alpha}{\alpha+2} +\ep} (\Omega), \qquad \forall \text{$\ep>0$, $\forall$ $\alpha>0$}; 
\\ 
\label{w1p}
f,g \in W^{1,p} (\Omega) &\not \Rightarrow& T \in C_{loc}^{0,\frac{p-1}{3p-1} +\ep}(\Omega), \qquad \forall \text{$\ep>0$} .
\end{eqnarray}

Note that \eqref{eqn:c-infty-23} recovers the aforementioned example appearing in \cite[\S 4]{LSW}. On the other hand, \eqref{bdry} improves the H\"older exponent when boundary regularity is taken into account. In \eqref{eqn:c21-12}, \eqref{eqn:calpha-calpha}, and \eqref{w1p} the regularity of the densities give a quantitative estimate on the H\"older exponent of the optimal map; in particular, \eqref{eqn:calpha-calpha} illustrates that one may degenerate the H\"older regularity of $T$ by reducing $\alpha$ and improves the bound derived from \eqref{ex1} (i.e. that $C^\alpha$ data implies a transport that is not more regular than $C^\alpha$). This improvement is generated by choosing nonparallel transport rays, see \S \ref{S1}.      


The situation for the quadratic cost, i.e. $c(x,y)=|x-y|^2/2$, is of a completely different nature. In this setting, if the densities are bounded away from zero and infinity on convex domains, then Caffarelli \cite{Ca91,Ca92,Ca96a} asserts that $T$ is H\"older continuous, cf. \cite[Example 4.5]{FGP}; moreover, if $f,g \in C^{k,\beta}$, then $T\in C^{k+1,\ep}$ for $k=0,1,\ldots$, and $\ep \in (0,\beta)$, cf. \eqref{eqn:c-infty-23}, \eqref{eqn:c21-12}, and \eqref{eqn:calpha-calpha} (see also Remark \ref{extens}). For general cost functions satisfying a strong form of the so-called Ma-Trudinger-Wang condition \cite{T1}, it was shown in \cite{Li} under mild conditions on the densities, that the transport map is H\"older continuous up to the boundary if the domains satisfy a certain type of convexity condition called $c$-convexity (see also the interior regularity result in \cite{Fi} for cost functions satisfying a weak form of the Ma-Trudinger-Wang condition). Indeed, the domain $\Omega$ above satisfies this convexity condition for a large class of cost functions. Thus, the Monge cost generates a significant loss of regularity compared to other cost functions for which a regularity theory is known. Nevertheless, it remains an open problem to determine sufficient conditions on the initial data that ensure the continuity of optimal maps in the classical Monge problem.

\section{Main result} \label{S1}
For the reader's convenience, we keep our notation as close as possible to \cite[\S 4]{LSW}. Our examples are built as follows.
The idea of the construction is to a priori fix the transport rays by specifying a family of line segments that may be thought of as the lines along which the optimal transport acts.
Let $\omega: [0,1] \to [0,1/2]$ be a function such that
\begin{equation}\label{defn:omega}
 \omega \in C^0([0,1]) \cap C^\infty ((0,1)), \quad \omega'(t)>0 \; \forall t \in (0,1), \quad \omega(0) = 0, \quad \omega(1) = \frac{1}{2}.
 \end{equation}
In our situation, we consider the following transport rays:
\begin{equation}\label{defn:la}
l_a =\{(x_1,x_2)\in \R^2: x_2= \omega(a)(x_1+a), \; x_1 \in (-a, 1)\} \qquad a\in(0,1).
\end{equation}
It is clear that the segments $l_a$ do not mutually intersect. The domain representing both source and target will be $\Delta \subseteq \R^2$, where 
\begin{equation}
\label{eqn:delta}
\Delta = \mbox{interior of the triangle with vertices $(-1,0)$, $(1,1)$, and $(1,0)$;}
\end{equation}
The initial and final density will have the form 
\begin{equation}\label{defn:fg}
 f(x) = 1_\Delta(x), \qquad g(x)= 1_\Delta(x) \big(1 + c( \zeta(x_1) + \eta(x_2) )\big), \qquad x = (x_1,x_2)\in \R^2,
\end{equation}
where 
\begin{equation}\label{defn:czetaeta}
\zeta(x_1) = -12 x_1^2 +12 x_1-2, \quad x_1\in [-1,1], \qquad \eta \in L^\infty(0,1) \cap C^\infty(0,1),
\end{equation}
and $c>0$ is chosen in such a way that $g$ is bounded away from zero:
\begin{equation}\label{defn:c}
c < \frac{1}{2\big(\|\zeta\|_{L^\infty(-1,1)}+ \|\eta\|_{L^\infty(0,1)}\big)}.
\end{equation}

Since we want the segments $l_a$ to be transport rays for the optimal map, the following mass balance condition for the region in the domain below each $l_a$ must be satisfied:
\begin{equation}\label{eqn:equality-mass}
 \int_{\Gamma_a} f = \int_{\Gamma_a} g  \qquad \forall a\in[0,1],
\end{equation}
where $\Gamma_a$ is the subgraph of $l_a$ in $\Delta$, namely
the triangle formed by $(-a,0)$, $(1,\omega(a)(1+a))$, and $(1,0)$; this mass constraint will determine the function $\eta$.

It is well known that the optimal map for the Monge cost is not unique. On the other hand, there exists a unique monotone optimal map \cite{FM} which in our setting can be shown to satisfy 
$$T(x) \in l_a \qquad \forall x\in l_a, \; a\in (0,1),$$
and
$$ (T(x)-T(y)) \cdot (x-y) \geq 0 \qquad \forall x,y \in l_a, \; a\in (0,1).
$$
From the construction, it follows that this map is pointwise defined in $\Delta$, and we investigate its regularity. In particular, it is not difficult to see that the origin is a fixed point for a suitable continuous extension and that 
the regularity of this map at the origin strongly depends on the behavior of $\omega$ (see Figure~\ref{xenon}); indeed, in the next section we show how suitable choices of $\omega$ will provide a collection of counterexamples to regularity. All of these results will be consequences of the following theorem. 
%


\begin{figure}[] \label{fig:controes}
\centering 
\includegraphics[scale= .8]{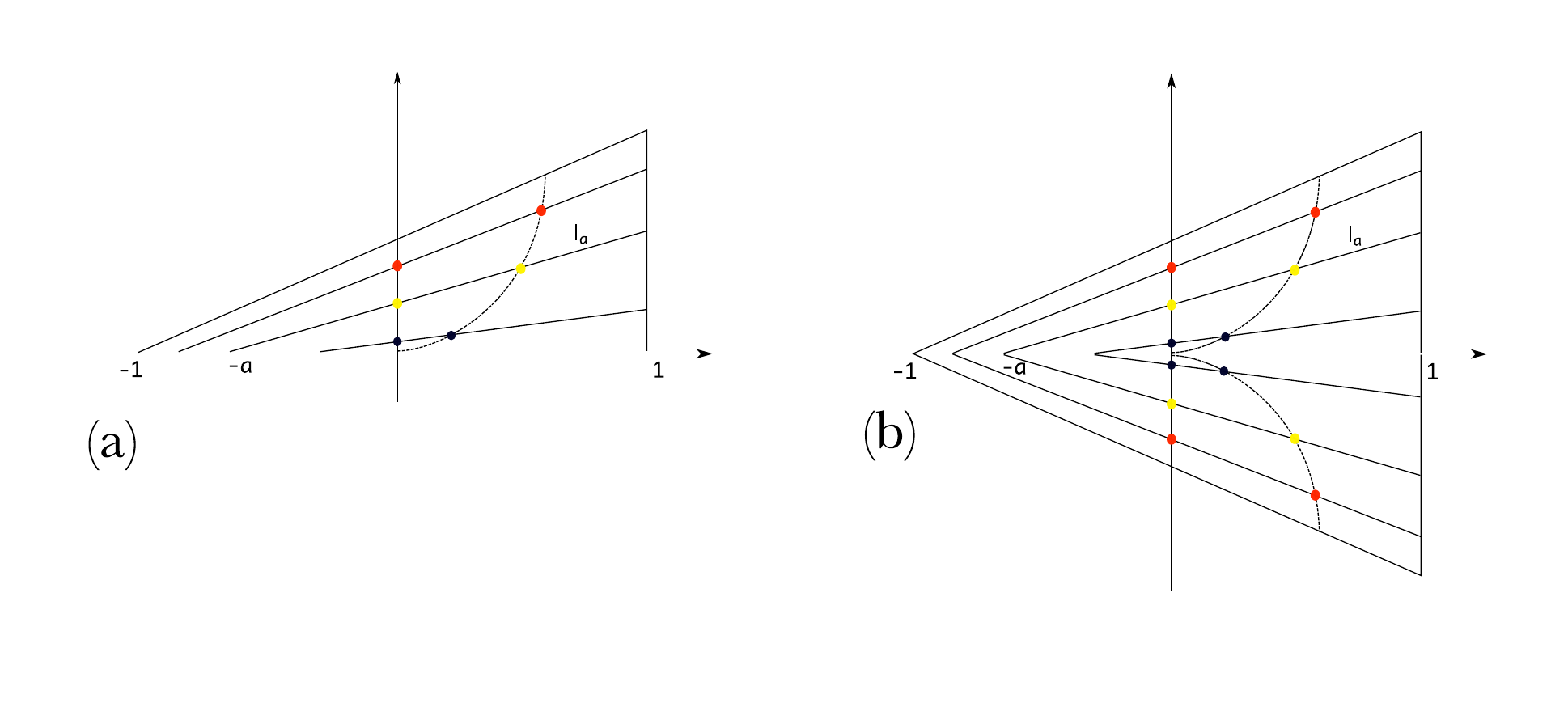}
\caption{Figure (a) depicts the nature of a boundary singularity when $\omega(a) = a^s$:
points on the vertical axis are mapped to points of the same color along the corresponding transport rays, and the horizontal component of the images approach the origin in a H\"older way, see Corollary \ref{cor:hold-dens}. Figure (b) illustrates how an interior singular point arises after reflecting the data.}
\label{xenon}
\end{figure} 

\begin{theorem}\label{thm:main}
Let $\omega$ be as in \eqref{defn:omega}, $\Delta \subseteq \R^2$ as in \eqref{eqn:delta}, and $\{l_a\}_{a\in(0,1)}$ the family of lines defined in \eqref{defn:la}. 
Let $f, g, c, \eta, \zeta$ be as in \eqref{defn:fg}, \eqref{defn:czetaeta}, \eqref{defn:c} and assume that \eqref{eqn:equality-mass} holds and $\displaystyle \lim_{t\to 0+} \eta(t) = 0$. Then the monotone optimal transport $T=(T_1,T_2)$ between $f$ and $g$ with respect to the Monge cost satisfies
\begin{equation}\label{ts:thm}
0< \liminf_{a\to 0+} \frac{T_1(\omega(a) a \, {\bf e}_2)}{a} \leq  \limsup_{a\to 0+} \frac{T_1(\omega(a) a \, {\bf e}_2)}{a} <\infty.
\end{equation}
\end{theorem}

To obtain counterexamples to interior regularity, it suffices to reflect the picture across the $x_1$-axis (see Figure \ref{xenon}); in this way we construct initial data for which the optimal map fails to be regular at an interior fixed point.
Let $\Delta' \subseteq \R^2$ be defined as 
\begin{equation}
\label{eqn:delta'}
\Delta': = \mbox{interior of the triangle with vertices $(-1,0)$, $(1,1)$, and $(1,-1)$,}
\end{equation}
and consider
\begin{equation}\label{defn:fg-refl}
 \tilde f(x) =\frac{ 1_{\Delta'}(x)}{2}, \qquad \tilde g(x)= \frac{ 1_{\Delta'}(x)}{2} \big(1 + c( \zeta(x_1) + \eta(|x_2|) )\big), \qquad x = (x_1,x_2)\in \R^2,
\end{equation}
where $c, \eta, \zeta$ are as in \eqref{defn:czetaeta} and \eqref{defn:c}.
By the uniqueness of monotone optimal mappings \cite{FM} it follows that the symmetric extension of $T$ across the $x_1$-axis is the corresponding optimal map between the densities $\tilde f$ and $\tilde g$. Note that the regularity of $\tilde g$ on the $x_1$-axis is affected by the choice of $\omega$ (via $\eta$).

\subsection{Obstruction to regularity with smooth data}
The first corollary appears in work of Li, Santambrogio, and Wang \cite[\S 4]{LSW}. It states that the monotone optimal transport map between smooth densities supported on convex domains is not, in general, more regular than a $C^{\frac{2}{3}}$ function in the interior. The proofs of the corollaries will employ Theorem \ref{thm:main} to describe the regularity of the transport map as well as an analysis lemma to guarantee the regularity of the target density (see Lemma \ref{lemma:reg-eta}).        

\begin{corollary}
Let $\Delta'$ be as  in \eqref{eqn:delta'}. There exists $g\in C^\infty (\Delta')$ bounded away from zero such that if $T$ is the monotone optimal transport map between $1_{\Delta'}$ and $g$, then
\begin{equation}\label{eqn:pippo1}
T\notin C^{\frac{2}{3}+\ep}(B_{1/4}(0)) \hskip .1in \text{for any $\ep>0$}.
\end{equation}
\end{corollary}

\begin{proof}
By selecting $\omega(a)=\sqrt{a}/2$, we choose $\eta \in C^\infty[-1,1]$ as in Lemma \ref{lemma:reg-eta} and $g$ as in \eqref{defn:fg-refl}. An application of Theorem \ref{ts:thm} yields
\begin{equation}\label{eqn:pippo}
 0<\liminf_{t\to 0+} \frac{T_1(t{\bf e}_2)}{t^{\frac{2}{3}}} \leq \limsup_{t\to 0+} \frac{T_1(t{\bf e}_2)}{t^{\frac{2}{3}}}<\infty.
\end{equation}
Thus, $T_1(t{\bf e}_2)$ can be extended continuously to $t=0$ by the second inequality in \eqref{eqn:pippo} with the value $0$. Finally, \eqref{eqn:pippo1} follows from the first inequality in \eqref{eqn:pippo}.     
\end{proof}

Next, we show that the monotone map between smooth data and convex domains is in general not better than $C^{\frac{1}{2}}$ up to the boundary.

\begin{corollary}
Let $\Delta$ be as  in \eqref{defn:fg}. There exists $g\in C^\infty (\overline{\Delta})$ bounded away from zero such that if $T$ is the monotone transport map between $1_{\Delta}$ and $g$, then
\begin{equation}\label{eqn:pippo4}
T\notin C^{\frac{1}{2}+\ep}(B_{1/4}(0)\cap \overline{\Delta}) \hskip .1in \text{for any $\ep>0$}.
\end{equation}
\end{corollary}

\begin{proof}
Let $\omega(a)=a/2$ and take $\eta \in C^\infty[0,1]$ as in Lemma \ref{lemma:reg-eta}. We select the smooth target density $g$ as in \eqref{defn:fg}, and apply Theorem \ref{ts:thm} to obtain
\begin{equation}\label{eqn:pippo3}
0<\liminf_{t\to 0+} \frac{T_1(t{\bf e}_2)}{t^{1/2}} \leq \limsup_{t\to 0+} \frac{T_1(t{\bf e}_2)}{t^{1/2}}<\infty.
\end{equation}
Thus, $T_1(t{\bf e}_2)$ can be extended continuously to $t=0$ by the second inequality in \eqref{eqn:pippo3} with the value $0$. Finally, \eqref{eqn:pippo4} follows from the first inequality in \eqref{eqn:pippo3}.  
\end{proof}

When $\omega(a)=a/2$, it can be seen from Lemma \ref{lemma:reg-eta} that $\eta \in C^\infty[0,1]$. In fact, by utilizing \eqref{eqn:2-deriv}, it is not difficult to show that $\eta$ is quadratic at the origin with a bounded left and right {third} derivative; thus, its even reflection is $C^{2,1}$, and this implies the following interior result:
\begin{corollary}
Let $\Delta'$ be as  in \eqref{eqn:delta'}. There exists $g\in C^{2,1}(\Delta')$ bounded away from zero such that if $T$ is the monotone transport map between $1_{\Delta'}$ and $g$, then
$$T\notin C^{\frac{1}{2}+\ep}(B_{1/4}(0)) \hskip .1in \text{for any $\ep>0$}.$$
\end{corollary}

\subsection{Obstruction to regularity with continuous data} In the next corollary, we show that one may degenerate the possible H\"older continuity of the monotone optimal transport with (degenerate) H\"older data.

\begin{corollary}\label{cor:hold-dens}
Let $\Delta'$ be as in \eqref{eqn:delta'}. For every $0<\alpha \leq 1$, there exists $g\in C^{\alpha} (\Delta')$ bounded away from zero such that if $T$ is the monotone optimal transport map between $1_{\Delta'}$ and $g$, then
$$T\notin C^{\frac{\alpha}{2+\alpha}+\ep}(B_{1/4}(0)) \hskip .1in \text{for any $\ep>0$}.$$
\end{corollary}

\begin{proof}
Let $\omega(a)=a^{s}/2$ where $s\ge2$ and choose $\eta$ as in Lemma~\ref{lemma:reg-eta}, so that $\eta \in C^{0, 2/s}[0,1]$ and $\eta(0)=0$. Since the even reflection is Lipschitz, it follows that $\eta(|x|) \in C^{0, 2/s}[-1,1]$. Set $\alpha:=2/s$ and select the $C^{0, \alpha}$ target density $g$ as in \eqref{defn:fg-refl}. An application of Theorem \ref{ts:thm} yields
\begin{equation*}
 0<\liminf_{t\to 0+} \frac{T_1(t{\bf e}_2)}{t^{\frac{\alpha}{2+\alpha}}} \leq \limsup_{t\to 0+} \frac{T_1(t{\bf e}_2)}{t^{\frac{\alpha}{2+\alpha}}}<\infty.
\end{equation*}
Thus, $T_1(t{\bf e}_2)$ can be extended continuously to $t=0$ with the value $0$, and this finishes the proof.         
\end{proof}

\begin{remark} \label{extens}
In Corollary~\ref{cor:hold-dens} one could also take $\alpha \in (0,\infty)$. In fact, it follows from Remark~\ref{rmk:even-other-alpha} and Theorem~\ref{thm:main}
that for every $k\in \N$, $\beta\in (0,1)$ there exists $g\in C^{k,\beta} (\Delta')$ bounded away from zero such that the monotone optimal transport map $T$ between $1_{\Delta'}$ and $g$ is not $ C^{\frac{k+\beta}{2+k+\beta}+\ep}$
for any $\ep>0$.
\end{remark}

\begin{remark}
One may not hope for a H\"older continuous map with only $C^0$ data. Indeed, this fact can readily be inferred from the example giving rise to \eqref{ex1}, and it can also be deduced from Theorem \ref{thm:main} and Lemma~\ref{lemma:reg-eta} by selecting the function $\omega(a) = e^{1-1/a}/2$. 
\end{remark}

\begin{remark}
It is not difficult to see that the function $g$ constructed in the proof of Corollary~\ref{cor:hold-dens}
is $W^{1,p}(\Delta')$ for every $p>(1-\alpha)^{-1}$. In particular, this shows that for every $p>1$, $$f,g \in W^{1,p} (\Delta') \not \Rightarrow T \in C_{loc}^{0,\frac{p-1}{3p-1} +\ep}(\Delta'), \qquad \forall \text{$\ep>0$} .$$ If $p=\infty$, it follows that the monotone optimal map between Lipschitz densities is not more regular than $C^{\frac{1}{3}}$ (this can also be deduced from Corollary~\ref{cor:hold-dens} by choosing $\alpha=1$).
\end{remark}

\section{Proof}
In the following lemmas, we determine the target density by imposing the mass balance condition \eqref{eqn:equality-mass} and build a Kantorovich potential. This construction yields the existence of an optimal transport map for the Monge cost between $f$ and $g$ whose transport rays are given by \eqref{defn:la}.

\begin{lemma}\label{lemma:reg-eta}
Let $\Delta$ be as  in \eqref{defn:fg}, $\omega$ as in \eqref{defn:omega}, $\{l_a\}_{a\in(0,1)}$ the family of lines defined in \eqref{defn:la}. Moreover, let $f, g, c, \eta, \zeta$ be as in \eqref{defn:fg}, \eqref{defn:czetaeta}, \eqref{defn:c}. Then the following statements hold:
\begin{enumerate}
\item[(i)] \label{lemma:cont}
\eqref{eqn:equality-mass} holds if and only if
\begin{equation}\label{ts:eta}
\eta(t) = \frac{d^2}{dt^2}\big(t^2 a^2(t) \big) \qquad \forall t\in (0,1), 
\end{equation}
where {$a(t) \in C^{0}[0,1] \cap C^\infty(0,1)$} is the unique solution of
\begin{equation}
\label{eqn:t-a}
t = \omega(a)(1+a).
\end{equation}
\vskip .2in 
\item[(ii)]
 If $\omega(a) = a^s/2$ for $a\in [0,1]$ and some $s\in [2,\infty)$, then $\eta \in C^{0,\frac{2}{s}}[0,1]$ with $\eta(0)=0$.
 Moreover, if $s=1/n$ for $n\in \N$ then $\eta \in C^{\infty}[0,1]$ and if $s=1/(2n)$ for $n\in \N$, then $\eta$ has a $C^\infty$ even extension to $[-1,1]$.
 \vskip .2in 
 \item[(iii)]\label{lemma:c0stat}
If {$\omega(a) = e^{1-1/a}/2$ for every $a\in (0,1]$, $\omega(0)=0$,} then $ \eta \in C^{0}[0,1]$ with $\eta(0)=0$.
\end{enumerate}
\end{lemma}

\begin{proof} $\mbox{}$
\\
{\it Proof of statement (i)}.
Note that
\eqref{eqn:equality-mass}
can be rewritten as
\begin{equation}\label{eqn:equality-mass-new2}
- \int_{\Gamma_a} \zeta(x_1) \, dx_1 \, dx_2 = \int_{\Gamma_a} \eta(x_2) \, dx_1 \, dx_2  \qquad a\in(0,1).
\end{equation}
Let $N\in C^\infty(0,1)$ be such that $N''(t) =\eta(t)$,
$N(0)=N'(0)=0$. To compute the second integral, we integrate first with respect to $x_2$
\begin{eqnarray}\nonumber
\int_{\Gamma_a} \eta(x_2) &=& \int_{-a}^1\int_0^{\omega(a)(x_1+a)} \eta(x_2)\, dx_2 \, dx_1
=\int_{-a}^1 N'(\omega(a)(x_1+a)) \, dx_1
\\
\label{eqn:equality-mass-boh}
&=&\frac{1}{\omega(a)}\int_{0}^{\omega(a)(1+a)} N'(x_1) \, dx_1
=\frac{ N(\omega(a)(1+a))}{\omega(a)}.
\end{eqnarray}
Similarly, we compute the left integral in \eqref{eqn:equality-mass-new2} by first integrating with respect to $x_1$ and then with respect to $x_2$. By noting that the function $Z(t) = -t^2 (t-1)^2$ satisfies $Z'' (t)= \zeta(t)$ and $Z(1)=Z'(1)=0$, it follows that
\begin{equation}\label{eqn:equality-mass-new}
- \int_{\Gamma_a} \zeta(x_1) \, dx_1 \, dx_2 
= -\int_0^{\omega(a)(a+1)} \int_{x_2/\omega(a)-a}^{1} \zeta(x_1) \, dx_1 \, dx_2 = \omega(a) a^2(a+1)^2.
\end{equation}
Thanks to \eqref{eqn:equality-mass-boh} and \eqref{eqn:equality-mass-new}, \eqref{eqn:equality-mass-new2} can be rewritten as
$$N(\omega(a)(1+a)) = \omega^2(a)a^2(a+1)^2, \qquad a\in (0,1);$$ i.e.,
\begin{equation}\label{defn:N}
N(t)= t^2 a^2(t), \qquad t\in (0,1),
\end{equation}
which is equivalent to \eqref{ts:eta}. Moreover, by \eqref{defn:omega} and the implicit function theorem, {the function $a: [0,1] \to [0,1]$ which solves \eqref{eqn:t-a} is monotone, continuous, and belongs to $C^\infty(0,1)$.}

{\it Proof of statement (ii)}.
Let $\omega(a) = a^s/2$ for $s \in (2,\infty)$. We are left to prove $\eta(0)=0$ and that $\eta$ is H\"older continuous at $0$. By the implicit function theorem, there exists a $C^\infty$ function $b$ defined in a neighborhood of $0$ so that
$$r= \frac{b(r)}{2^{1/s}}(1+b(r))^{1/s}$$
for every $r\geq 0$ sufficiently small. Therefore, $b(t^{1/s})$ is a solution to $$t= \frac{b^s(t^{1/s})}{2} (1+b(t^{1/s})),$$ and hence $a(t) = b( t^{1/s})$ for every $t$ sufficiently small.
Therefore the function $t^2 a^2(t) = t^2 b^2(t^{1/s})$ belongs to $C^{2,2/s}[0,1]$. Indeed, an explicit computation shows that away from the origin, the second derivative is a smooth function of $t^{1/s}$ through the formula:
\begin{equation}\label{eqn:2-deriv}
\frac{d^2}{dt^2} [t^2b(t^{1/s})]=
2b^2(t^{1/s}) + \Big(\frac{6}{s}+\frac{2}{s^2}\Big) t^{1/s} b( t^{1/s}) b'( t^{1/s})
+\frac{2}{s^2} t^{2/s} \big( [b'( t^{1/s})]^2 + b( t^{1/s}) b''( t^{1/s}) \big).
\end{equation}
By computing the third derivative and integrating it between any two points, it follows that each term in this expression is $C^{2/s}$. We show this regularity for the first term in \eqref{eqn:2-deriv} (the argument is similar for the other terms). Since $b$ is Lipschitz and $b(0)=0$, we have that $b(r) \leq r\|b'\|_{L^\infty((0,1))}$ for every $r\in[0,1]$; hence for every $t_1,t_2\in [0,1]$ with $t_1<t_2$
\begin{equation*}
\begin{split}
b^2(t_1^{1/s}) - b^2(t_2^{1/s}) 
&= 
\int_{t_1}^{t_2} \frac{d}{dt}[b^2(t^{1/s})] \, dt
\leq\frac{2}{s} \int_{t_1}^{t_2} b(t^{1/s}) b'(t^{1/s})  t^{1/s-1}\, dt
\\
&\leq
\frac{\ 2|b'\|^2_{\infty}}{s} \int_{t_1}^{t_2} t^{2/s-1}\, dt 
= \|b'\|^2_{\infty} \big({t_2}^{2/s}- {t_1}^{2/s}\big) 
\leq \|b'\|^2_{\infty} \big({t_2}- {t_1}\big)^{2/s} .
\end{split}
\end{equation*}
Since $b(0)=0$, \eqref{eqn:2-deriv} implies $\eta(0)=0$.  
Moreover, if $s=1/n$ for some $n\in \N$, then $a(t) = b(t^n)$ is $C^\infty$ up to $t=0$.
Similarly, if $s=1/(2n)$, then $t^2 a^2(t) = t^2 b^2(t^{2n}) \in C^\infty[0,1]$ with an even extension for $t\in [-1,1]$ (since the function $t^2 a^2(t)$ depends smoothly on $t^2$). 

{\it Proof of statement (iii)}.
We apply the same line of reasoning as before by letting $\omega(a) = e^{1-1/a}/2$: consider the equation $$t=\frac{e^{1-1/a}}{2}(1+a),$$ which is equivalent to
\begin{equation*} \label{dre}
r = \frac{b(r)}{1 - b(r) \left(\log (1+b(r)) +1-\log{2}\right)},
\end{equation*}
where $r=\frac{-1}{\log(t)}$. By the implicit function theorem, there exists a smooth solution $b=b(r)$ in a neighborhood of the origin. Thus, $a(t)=b(r(t))= b(-1 /\log (t))$, and this implies that the function $t^2 a^2(t) = t^2 b^2(-1 /\log (t))$ belongs to $C^{2}[0,1]$; since $b(0)=0$, we deduce $\eta(0)=0$.  
\end{proof}

\begin{remark}\label{rmk:even-other-alpha}
In Lemma~\ref{lemma:reg-eta} (ii),  one could also take $s\in (0,2]$ in order to obtain a more regular $\eta$. More precisely, if $\omega(a) = a^s/2$ with $s\in (0,\infty)$, then $\eta(|\cdot|) \in C^{k,\beta}([-1,1])$, where $k$ is the integer part of $\frac{2}{s}$ and $\beta=\frac{2}{s}-k$. 
Indeed, the regularity of $\eta$ is determined by the regularity of $N$ as in \eqref{defn:N}; thanks to Lemma~\ref{lemma:reg-eta} (i), it is enough to check the regularity of $\eta$ at the origin. This can be made explicit by noting that $a(t)= O(t^{1/s})$ close to $t=0$, see \eqref{eqn:t-a}. It follows that $N(t) = O(t^{2+2/s})$. To make this argument rigorous, one can rewrite $N(t)=t^2 b^2(t^{1/s})$ with $b$ smooth up to the origin and explicitly compute the derivatives of this expression noting that many vanish at the origin. We also remark that in Lemma~\ref{lemma:reg-eta} (iii), a modulus of continuity of $\eta$ is given by the function {$-1/\log(\cdot)$} (up to a constant).
\end{remark}

\begin{remark}
Slight generalizations of the construction in this paper do not lead to better counterexamples. For instance, the particular choice of $\zeta$ is made in such a way that $\zeta(0) \neq 0$, a fact which is needed in the proof of Theorem \ref{thm:main}. This implies that the function $Z\in C^\infty(0,1)$ satisfying $Z'' (t)= \zeta(t)$ and $Z(1)=Z'(1)=0$, is at most quadratic at the origin.
On the other hand, this limits the regularity of $\eta$ in Lemma~\ref{lemma:reg-eta}.
Similarly, altering $f$ to have the same structure as $g$, namely
$$ f(x)= 1_\Delta(x) \big(1 + c( \tilde \zeta(x_1) + \tilde \eta(x_2) )\big), \qquad x = (x_1,x_2)\in \R^2,$$
does not lead to better counterexamples, as can be seen from the proof of Lemma~\ref{lemma:reg-eta}.
\end{remark}

\begin{lemma} \label{potent}
{Let $\Delta$ be as  in \eqref{defn:fg}, $\omega$ as in \eqref{defn:omega}, and $\{l_a\}_{a\in(0,1)}$ the family of lines defined in \eqref{defn:la}.}
Then there exists $u \in C^{0,1}(\overline{\Delta})\cap C^2(\Delta)$ such that
$$|u(x)-u(y)| \leq |x-y|, \qquad x,y\in\Delta,$$
with equality if and only if $x,y \in l_a$ for some $a\in (0,1)$.
\end{lemma}
\begin{proof}
For every $x \in \Delta$, let $a(x)$ be such that the unique line passing through $x$ is $l_{a(x)}$. Note that $a(x)$ is obtained by solving 
\begin{equation} \label{eqn:a-solves}
x_2= \omega(a)(x_1+a);
\end{equation}
the assumptions on $\omega$ imply $a \in C^1(\Delta)$ (via the implicit function theorem). Next, consider the unit vector field $b \in C^1(\Delta; \R^2)$ which points in the direction of $l_{a(x)}$ for each $x \in \Delta$, i.e. $$b(x) = \frac{1}{(1+\omega^2(a(x)))^{1/2}} ({\bf e}_1+ \omega(a(x)){\bf e}_2)  \qquad \forall x\in \Delta. $$
The statement of the lemma is equivalent to the existence of a function $u\in C^1(\Delta)$ such that 
\begin{equation} \label{eqn:b-grad}
Du(x) = b(x) \qquad \forall x\in \Delta.
\end{equation}
Indeed, if \eqref{eqn:b-grad} holds, then $||Du||_{L^\infty((0,1))} \leq 1$ and $\partial_t[u(-a{\bf e}_1+t b(x))]=|b(x)|^2=1$ for every $x \in \Delta$ and $t>0$; since $-a{\bf e}_1+t b(x)$ is a parametrization of the line $l_{a(x)}$ with $0<t<(a+1)(1+\omega^2(a))^{1/2}$, this shows that $u$ is linear with slope $1$ when restricted to $l_{a(x)}$.
On the other hand if the statement of the lemma holds, then $|Du|\leq 1$ and to satisfy the equality cases, the gradient of $u$ must point in the direction of $l_{a(x)}$ for every $x\in \Delta$. As $\Delta$ is simply connected, \eqref{eqn:b-grad} is equivalent to showing that $b$ is an irrotational vector field, i.e. 
$$0 = \partial_1 b_2(x) - \partial_2 b_1(x) \qquad \forall x\in \Delta,$$ or equivalently
\begin{equation}\label{eqn:to-be-a-gradient}
0 = \frac{\omega(a(x))\omega'(a(x))}{(1+\omega(a(x)))^{3/2}} \Big(\partial_2 a(x) + \frac{\partial_1a(x)}{\omega(a(x))}\Big)  \qquad \forall x\in \Delta.
\end{equation}
By applying $\partial_1$ and $\partial_2$ to $\eqref{eqn:a-solves}$, it follows that
\begin{equation}\label{eqn:a-diff1}
 0= \frac{\omega'(a(x))x_2}{\omega(a(x))}  \partial_1a(x) + \omega(a(x)) (\partial_1a(x)+1) \qquad \forall x\in \Delta,
\end{equation}
\begin{equation}\label{eqn:a-diff2}
 1 = \frac{\omega'(a(x))x_2}{\omega(a(x))}  \partial_2a(x) + \omega(a(x))\partial_2a(x) \qquad \forall x\in \Delta.
\end{equation}
Multiplying \eqref{eqn:a-diff2} by $\omega(a(x))$ and adding \eqref{eqn:a-diff1} yields 
$$0 = \partial_2 a(x) + \frac{\partial_1a(x)}{\omega(a(x))}  \qquad \forall x\in \Delta,$$
which proves \eqref{eqn:to-be-a-gradient}. To conclude, we note that $u$ admits a Lipschitz extension to $\overline{\Delta}$.  
\end{proof}

%
%
%
%
%

\begin{proof}[Proof of Theorem~\ref{thm:main}]
By the construction in \cite{CFM, TW} and \cite{FM}, there exists a unique measure preserving map $T:\Delta \to \Delta$ between $f$ and $g$ such that $x$ and $T(x)$ are contained on a common line $l_a$ for all $x \in \Delta$. Moreover, by Lemma \ref{potent} (which exploits the construction of the Kantorovich potential), one readily obtains that $T$ is an optimal mapping in the Monge problem, see e.g. \cite[Lemma 7]{LSW}. Since $T$ is measure preserving, the construction in \cite{CFM, TW} implies for every $a\in (0,1)$
\begin{equation}\label{eqn:mist}
  \lim_{\delta\to 0} \frac{1}{\delta} \int_{(\Gamma_{a+\delta}\setminus\Gamma_a )\cap \{x_1<0\}} f(x) \, dx = 
    \lim_{\delta\to 0} \frac{1}{\delta} \int_{(\Gamma_{a+\delta}\setminus\Gamma_a )\cap \{x_1<T_1(a \omega(a){\bf e}_2)\}} g(x) \, dx.
 \end{equation}
Note that $$\frac{d}{d\delta} \Big|_{\delta=0}\Big[\omega(a+\delta)(x_1+a+\delta)\Big]=\omega'(a)\Big(x_1+a+\frac{\omega(a)}{\omega'(a)}\Big);$$ thus, \eqref{eqn:mist} may be rewritten as
\begin{equation} \label{finito}
 \int_{-a}^0 \Big(x_1+a +\frac{\omega(a)}{\omega'(a)}\Big) \, dx_1 =
 \int_{-a}^{T_1(a \omega(a){\bf e}_2)} \Big(x_1+a +\frac{\omega(a)}{\omega'(a)}\Big) \Big(1+c\big(\zeta(x_1)+ \eta(\omega(a) (x_1+a))\big)\Big) \, dx_1,
\end{equation}
(after canceling $\omega'(a)$ from both sides).
Next, set
$$
z(a):= T_1(a\omega(a) {\bf e}_2), \qquad a\in (0,1).$$ To prove the second inequality in \eqref{ts:thm}, 
note that by \eqref{defn:c}, the second term of the integrand on the right-hand side of \eqref{finito} is greater than $1/2$ for every $x_1\in (-1,1)$; therefore
$$ \frac 1 2
 \int_{-a}^{z(a)} \Big(x_1+a +\frac{\omega(a)}{\omega'(a)}\Big) \, dx_1
\leq
\int_{-a}^0 \Big(x_1+a +\frac{\omega(a)}{\omega'(a)}\Big) \, dx_1.
 $$
By evaluating the two integrals we obtain 
 $$
\frac 1 2 \Big( \frac{(z(a)+a)^2}{2} +\frac{\omega(a)(z(a)+a)}{\omega'(a)} \Big)
\leq
\frac{a^2}{2} +\frac{\omega(a)a}{\omega'(a)}
\leq
\frac 1 2 \Big( \frac{(2a)^2}{2} +\frac{\omega(a) 2a}{\omega'(a)}
\Big).
 $$
Since the function $t\to t^2/2+t\omega(a)/\omega'(a)$ is increasing on $[0,\infty)$, it follows that
 $${z(a)+a} \leq 2a,
 $$
which proves that the $\limsup$ in \eqref{ts:thm} is bounded by $1$. Next, we prove the first inequality in \eqref{ts:thm}.
By the properties of $\zeta$ and $\eta$ it follows that there exists a constant $0<c_0<1$ such that $\zeta \leq -2c_0$ and $\eta \leq c_0$ in a neighborhood of the origin.
Hence, since $z(a) \to 0$ as $a \to 0$, for $a$ sufficiently small,
$$ \int_{-a}^0 \Big(x_1+a +\frac{\omega(a)}{\omega'(a)}\Big) \, dx_1 \leq
(1-c_0)
 \int_{-a}^{z(a)} \Big(x_1+a +\frac{\omega(a)}{\omega'(a)}\Big) \, dx_1.
 $$
Evaluating the two integrals yields
 $$
\frac{a^2}{2} +\frac{\omega(a)a}{\omega'(a)}\leq(1-c_0) \Big( \frac{(z+a)^2}{2} +\frac{\omega(a)(z(a)+a)}{\omega'(a)} \Big).
 $$
 On the other hand, note that 
 $$
 (1-c_0) \Big( \frac{a^2}{2(1-c_0)} +\frac{\omega(a)a}{\omega'(a) (1-c_0)^{1/2}} \Big)
 \leq \frac{a^2}{2} +\frac{\omega(a)a}{\omega'(a)}.$$
Since the function $t\to t^2/2+t\omega(a)/\omega'(a)$ is increasing on $[0,\infty)$,
 we deduce
 $$\frac{a}{(1-c_0)^{1/2}} \leq z(a)+a.
 $$
Subtracting $a$ from both sides yields that the $\liminf$ in \eqref{ts:thm} is bounded below by $$\frac 1 {(1-c_0)^{1/2}}-1>0.$$
\end{proof}

\emph{Acknowledgments.}
We thank Filippo Santambrogio for his remarks on a preliminary version of the paper and for pointing out the example leading to \eqref{ex1}. This work was completed at the Mathematical Sciences Research Institute in Berkeley, California, while the first author was a Program Associate and the second author a Huneke Postdoctoral Scholar. The warm hospitality of the institute is kindly acknowledged.

\def\cprime{$'$} \def\cprime{$'$}

\signei

\signdm

\end{document}